\documentclass[11pt]{amsart}
\usepackage{amstext}
\usepackage{amsthm}
\usepackage{amsopn}
\usepackage{amsfonts}
\usepackage{amsmath}
\usepackage{amssymb}
\usepackage{amscd}
\usepackage{pgf,pgfnodes}
\usepackage{wrapfig}
\usepackage{amscd}

\newtheorem{theorem}{Theorem}[section]

\newtheorem{proposition}[theorem]{Proposition}

\newtheorem{conj}{Conjecture}

\theoremstyle{definition}

\theoremstyle{remark}
\newtheorem{remark}[theorem]{Remark}

\numberwithin{equation}{section}

\def\q{\quad}
\def\Qbar{\text{\sl Q\kern-.45em{\vrule height.63em width.05em
depth-.033em}}~}
\def\qbar{{{\scriptstyle Q}\kern-.45em{\vrule height.41em width.035em
depth-.03em}}~}
\def\Cbar{\text{\sl C\kern-.35em{\vrule height.63em width.05em
depth-.033em}}~}
\def\cbar{{{\scriptstyle C}\kern-.41em{\vrule height.42em width.035em
depth-.03em}}~}
\def\ibid{\hbox to .5truein{\hrulefill}}
\def\IB{\text{{\rm I}\kern-.13em{\rm B}}}

\def\Z{\Bbb Z}
\def\IR{\text{{\rm I}\kern-.13em{\rm R}}}

\def\twoheaddown{\downarrow\kern-0.78em\raise0.25em\hbox{$\downarrow$
}}
\def\headtaildown{\downarrow\kern-0.79em\raise
0.5em\hbox{$\ssize\curlyvee$}}
\def\lda{\downarrow\kern-.655em\raise0.5em\hbox{$\vert$
}}

\def\bs{\backslash}
\def\la{\langle}
\def\ra{\rangle}
\def\wt{\widetilde}

\def\os{\overset}
\def\us{\underset}

\def\Hom{\text{\rm Hom}}
\def\Det{\text{\rm Det}}
\def\Gal{\text{\rm Gal}}
\def\ind{\text{\rm ind}}

\def\fP{{\frak p}}

\def\dfrac#1#2{{\displaystyle {{#1}}\over{\displaystyle {#2}}}}

\def\suba{\subset\kern-.6em\raisebox{-.27em}{$\rightarrow$}}

\begin{document}

\title[Evidence toward a $2$-adic equivariant ``main conjecture'']{Numerical evidence toward a $2$-adic equivariant ``main conjecture''}
\author[X.-F. Roblot]{Xavier-Fran\c{c}ois Roblot$^*$}
\thanks{$^*$supported in part by the ANR (projet \textit{AlgoL})}
\address{Institut Camille Jordan, Universit\'e de Lyon, Universit\'e
  Lyon~1, CNRS -- UMR 5208, 43 blvd du 11 Novembre 1918, 69622 Villeurbanne Cedex, France}
\author[A. Weiss]{Alfred Weiss$^\dagger$}
\thanks{$^\dagger$financial support from the NSERC is gratefully aknowledged}
\address{Departement of Mathematics, University of Alberta, Edmonton, AB T6G 2G1, Canada}
\date{}

\maketitle
\section{The conjecture}

Let $K$ be a totally real finite Galois extension of $\Qbar$ with Galois group $G$ dihedral of order $8,$ and suppose that $\sqrt 2$ is not in $K.$  Fix a finite set $S$ of primes of $\Qbar$ including $2,\infty  $ and all primes that ramify in $K.$ Let $C$ be the cyclic subgroup of $G$ of order $4$ and $F$ the fixed field of $C$ acting on $K.$  Fix a $2$-adic unit $u\equiv 5 \mod 8\Z_2.$

Write $L_F(s,\chi )$ for the $2$-adic $L$-functions, normalized as in \cite{W}, of the $2$-adic characters $\chi$ of $C$ or, equivalently by class field theory, of the corresponding $2$-adic primitive ray class characters.  We always work with their $S${\it -truncated} forms
$$
L_{F,S}(s,\chi  ) = L_F(s,\chi) \prod_{\fP} \left(1-\frac{\chi ({\fP})}{N(\fP)} \langle N({\fP}) \rangle^{1-s}\right)
$$
where ${\fP}$ runs through all primes of $F$ above $S\bs\{2,\infty \}$, and $\langle\,\rangle : \mathbb{Z}_2^\times \to 1 + 4\mathbb{Z}_2$ is the unique function with $\langle x\rangle x^{-1} \in \{-1, 1\}$ for all $x$. Now our interest is in the $2$-adic function
$$
f_1(s) = \frac{\rho_{F,S}\log(u)}{8(u^{1-s}-1)} +\frac 18 \left(L_{F,S}(s,1) + L_{F,S}(s,\beta^2) - 2L_{F,S}(s,\beta)\right)
$$
where $\beta  $ is a faithful irreducible $2$-adic character of $C$ and
$$
\rho_{F,S} = \lim_{s\to1} (s-1)L_{F,S}(s,1).
$$
It follows from known results that $\frac 12 \rho_{F,S}\in \Z_2$ and that $f_1(s)$ is an {\it Iwasawa analytic} function of $s\in \Z_2,$ in the sense of \cite{R}.  This means that there is a unique power series $F_1(T)\in \Z_2[[T]]$ so that
$$
F_1(u^n-1) = f_1(1-n)\q\text{\rm for}\q n=1,2,3,\dots\,.
$$
The conjecture we want to test is 
\begin{conj}\label{conj1}
$$
\frac 12 \rho_{F,S} \in 4\Z_2\q\text{\rm and}\q F_1(T)\in 4\Z_2[[T]].
$$
\end{conj}
Testing the conjecture amounts to calculating $\frac 12 \rho_{F,S}$ and (many of) the power series coefficients of
$$
F_1(T) =\sum^\infty_{j=1} x_j T^{j-1}
$$
modulo $4\Z_2.$  Were the conjecture false we would expect to find a counterexample in this way.

The idea of the calculation is, roughly, to  express the coefficients of the power series $F_1(T)$ as integrals over suitable $2$-adic continuous functions with respect to the measures used to construct the $2$-adic $L$-functions. 

The conjecture has been tested for $60$ fields $K$ determined by the size of their discriminant and the splitting of $2$ in the field $F$. For this purpose, it is convenient to replace the datum $K$ by $F$ together with the ray class characters of $F$ which determine $K$ (cf \S 5).  A description of the results is in \S 6: they are affirmative.

Where does $f_1(s)$ come from?  It is an example which arises by attempting to refine the Main Conjecture of Iwasawa theory. This connection will be discussed next in order to prove that $F_1(T)$ is in $\Z_2[[T]].$

\section{The motivation}
The Main Conjecture of classical Iwasawa theory was proved by Wiles \cite{W} for odd prime numbers $\ell.$  More recently \cite{RW2}, an equivariant ``main conjecture'' has been proposed, which would both generalize and refine the classical one for the same $\ell.$  When a certain $\mu$-invariant vanishes, as is expected for odd $\ell$ (by a conjecture of Iwasawa), this equivariant ``main conjecture'', up to its uniqueness assertion, depends only on properties of $\ell$-adic $L$-functions, by Theorem A of \cite{RW3}.

The point is that it is possible to numerically test this Theorem A property of $\ell$-adic $L$-functions, at least in simple special cases when it may be expressed in terms of congruences and the special values of these $L$-functions can be computed.  The conjecture of \S1 is perhaps the simplest non-abelian example when this happens, but with the price of taking $\ell=2.$ Although there are some uncertainties about the formulation of the ``main conjecture'' for $\ell=2,$ partly because \cite{W} no longer applies, it seems clearer what the $2$-adic analogue of the Theorem~A properties of $L$-functions should be, in view of their ``extra'' $2$-power divisibilities \cite{DR}.

More precisely, let $L_{k,S}\in \Hom^*\left(R_\ell (G_\infty),\mathcal{Q}^{\,c}(\Gamma _k)^\times\right)$ be the ``power series'' valued function of $\ell$-adic characters $\chi $ of $G_\infty = \,\Gal(K_\infty  /k)$ defined in \S4 of \cite{RW2}.  This is made from the values of $\ell$-adic $L$-functions by viewing them as a quotient of Iwasawa analytic functions, by the proof of Proposition~11 in \cite{RW2}.  When $\ell \ne 2,$ the vanishing of the $\mu$-invariant mentioned above means precisely that the coefficients of these power series have no nontrivial common divisor; and the Theorem~A property of $L$-functions is that then $L_{k,S}$ is in $\text{\rm Det}\left(K_1(\Lambda (G_\infty)_\bullet)\right)$ (see next section for
precise definitions).

When $\ell =2,$ we can still form $L_{k,S},$ but now its values at characters $\chi$ of degree $1$ have
numerators divisible by $2^{[k:\Qbar]},$ because of (4.8), (4.9) of \cite{R}.  Define
$$
\wt L_{k,S}(\chi) = 2^{-[k:\Qbar]\chi(1)}L_{k,S}(\chi)
$$
for all $2$-adic characters $\chi $ of $G_\infty ,$ so that the deflation and restriction property of Proposition~12 of \cite{RW2} are maintained.  Then the analogous coprimality condition on coefficients of numerator, denominator of the values $\wt L_{k,S}(\chi)$ will be referred to as vanishing of the $\wt \mu  $-invariant of $K_\infty  /k$:  the Theorem~A property we want to test is therefore \begin{conj}\label{conj2}
$$
\wt L_{k,S} \q\text{\rm is in} \q \text{\rm Det}\left(K_1\big(\Lambda (G_\infty)_\bullet)\right).
$$
\end{conj}

\begin{remark} a) When Conjecture 2 holds, then $\wt L_{_{k,S}}(\chi)$ is in $\Lambda ^c(\Gamma  _k)^\times _\bullet$ for all $\chi  \in R_2(G_\infty  ),$ implying the vanishing of the $\wt\mu$-invariant of $K_\infty  /k\,.$

b) For $\ell \ne 2,$ some cases of the equivariant ``main conjecture'' have recently been proved (\cite{RW}).
\end{remark}

\section{Interpreting Conjecture~\ref{conj2} as a congruence}

We now specialize to the situation of \S1, so use the notation of its first paragraph, in order to exhibit a congruence equivalent to Conjecture~\ref{conj2} (see Figure~1).

\begin{figure}[htp]
\begin{center}
    \begin{pgfpicture}{-.5cm}{-.3cm}{4cm}{4.3cm}
     \pgfsetxvec{\pgfxy(.6,0)}
    \pgfsetyvec{\pgfxy(0,.8)}
   \pgfnodebox{Q}[virtual]{\pgfxy(2,0)}{$\Qbar$}{2pt}{2pt} 
   \pgfnodebox{K}[virtual]{\pgfxy(0,2.5)}{$K$}{2pt}{2pt} 
   \pgfnodebox{Kinf}[virtual]{\pgfxy(4,5)}{$K_\infty$}{2pt}{2pt} 
   \pgfnodebox{Qinf}[virtual]{\pgfxy(6,2.5)}{$\Qbar_\infty$}{2pt}{2pt} 
   \pgfnodeconnline{Q}{K}
   \pgfnodeconnline{K}{Kinf}
   \pgfnodeconnline{Q}{Qinf}
   \pgfnodeconnline{Qinf}{Kinf}
   \pgfnodeconnline{Kinf}{Q}
   \pgfnodelabel{Q}{K}[0.5][8pt]{\pgfbox[center,top]{$G$}}
   \pgfnodelabel{K}{Kinf}[0.5][8pt]{\pgfbox[center,base]{$\Gamma$}}
   \pgfnodelabel{Kinf}{Qinf}[0.5][8pt]{\pgfbox[center,base]{$H$}}
   \pgfnodelabel{Qinf}{Q}[0.5][8pt]{\pgfbox[center,top]{$\Gamma_\Qbar$}}
   \pgfnodelabel{Q}{Kinf}[0.6][8pt]{\pgfbox[right,top]{$G_\infty$}}
  \end{pgfpicture}
\end{center}
\caption{}
\end{figure}

Let $\Qbar_\infty $ be the cyclotomic $\Z_2$-extension of $\Qbar,$ i.e. the maximal totally real subfield of the field obtained from $\Qbar$ by adjoining all $2$-power roots of unity, and set $\Gamma _{\Qbar} = \;\Gal\,(\Qbar_\infty /\Qbar) \simeq \Z_2.$ Let $K_\infty = K\Qbar_\infty ,$ noting that $K\cap \Qbar_\infty =\Qbar$ follows from $\sqrt 2\notin K,$ and set $G_\infty =\;\Gal(K_\infty /\Qbar)$.  Defining $\Gamma = \,\ker(G_\infty \to G)$, $H=\,\ker\,(G_\infty \to \Gamma _\Qbar)$, we now have $H\hookrightarrow G_\infty \twoheadrightarrow \Gamma_\Qbar$ in the notation of \cite{RW2}.

Since $G_\infty =\Gamma \times H$ with $\Gamma \simeq \Gamma _\Qbar$ and $H\simeq G$ dihedral of order $8$ we can understand the structure of
$$
\Lambda    (G_\infty)_\bullet = \Lambda (\Gamma)_\bullet \otimes_{\Z_2} \Z_2[H] = \Lambda  (\Gamma )_\bullet [H]
$$
where $\bullet$ means ``invert all elements of $\Lambda (\Gamma )\bs 2\Lambda (\Gamma).$''

\smallskip

Namely, choose $\sigma,\tau$ in $G$ so that $C=\la \tau \ra$ with $\sigma^2=1,$ $\sigma \tau \sigma ^{-1}=\tau ^{-1}$ and extend them to $K_\infty, $ with  trivial action on $\Qbar_\infty, $ to get $s,t$ respectively. Then the abelianization of $H$ is $H^{ab} = H/\la t^2\ra$ and we get a pullback diagram
$$
\begin{CD}
\Lambda(G_\infty)_\bullet = \Lambda(\Gamma)_\bullet [H] @>>> \left(\Lambda(\Gamma)_\bullet (\zeta_4)\right)*\la s\ra 
\\
@VVV @VVV
\\
\Gamma(G^{ab}_\infty)_\bullet = \Lambda  (\Gamma)_\bullet[H^{ab}] @>>> \Lambda  (\Gamma)_\bullet[H^{ab}]/2\Lambda  (\Gamma)_\bullet [H^{ab}] 
\end{CD}
$$
where the upper right term is the crossed product order with $\Lambda(\Gamma )_\bullet$-basis $1,\zeta_4,\wt s,\zeta_4\wt s$ with $\zeta^2_4 = -1,$\; $\wt s\,^2=1,$\; $\wt s\zeta_4 = \zeta^{-1}_4\wt s = -\zeta_4\wt s$ and the top map takes $t,s$ to $\zeta_4,\wt s$ respectively, the right map takes $\zeta_4,\wt s$ to $t^{ab},s^{ab}.$ This diagram originates in the pullback diagram for the cyclic group $\la t\ra$ of order $4,$ then going to the dihedral group ring $\Z_2[H]$ by incorporating the action of $s,$ and finally applying $\Lambda (\Gamma)_\bullet\otimes_{\Z_2}-.$

We now turn to getting the first version of our congruence in terms of the pullback diagram above.  This is possible since $R^\times \to K_1(R)$ is surjective for all the rings considered there.  We also simplify notation a little by setting ${\frak A} = \left(\Lambda(\Gamma)_\bullet (\zeta_4)\right)*\la s\ra$ and writing $\wt L_{k,S}$ as $\wt L_{K_\infty /k}\,,$ because we will now have to vary the fields and $S$ is fixed anyway.  The dihedral group $G$ has $4$ degree $1$ irreducible characters $1,\eta ,\nu,\eta \nu$ with $\eta (\tau )=1,$\; $\nu(\sigma )=1$ and a unique degree $2$ irreducible $\alpha ,$ which we view as characters of $G_\infty $ by inflation.

\begin{proposition}\label{Prop1}
Let $K^{ab}_\infty  $ be the fixed field of $\la t^2\ra,$ hence $\Gal(K^{ab}_\infty  /\Qbar ) = G^{ab}_\infty$.  Then
\begin{enumerate}
\item[\rm a)] $ \wt L_{K^{ab}_\infty  /\Qbar}=\,\Det(\wt \Theta^{\,ab})$ for some $\wt\Theta^{\,ab} \in \Lambda (G^{ab}_\infty)^\times_\bullet$ 
\item[\rm b)] $\wt L_{K_\infty  /\Qbar} \in \Det(K_1\left(\Lambda(G_\infty)_\bullet\right))$ if, and only if, any $y \in {\frak A}$ mapping to $\wt\Theta^{\,ab} \mod 2$ in $\Lambda(G^{ab}_\infty)_\bullet/2\Lambda (G^{ab}_\infty)_\bullet$ has
$$
nr(y) \equiv \wt L_{K_\infty  /\Qbar}(\alpha) \mod 4\Lambda(\Gamma_\Qbar)_\bullet 
$$
\end{enumerate}
where $nr$ is the reduced norm of \textup{(}the total ring of fractions of\textup{)} ${\frak A}$ to its centre $\Lambda (\Gamma)_\bullet$ and we identify $\Lambda (\Gamma )_\bullet$ with $\Lambda(\Gamma_{\Qbar})_\bullet$ via $\Gamma \overset{\simeq}\rightarrow \Gamma_\Qbar.$
\end{proposition}

\begin{proof}
a) The vanishing of $\wt\mu $ for $K_\infty  /\Qbar,$ in the sense of \S2, is known by \cite{FW}, i.e. $\wt L_{K^{ab}_\infty /\Qbar}(\chi )$ is a unit in $\Lambda(\Gamma _\Qbar)_\bullet$ for all $2$-adic characters $\chi $ of $G^{ab}_\infty$.  By the proof of Theorem~9 in \cite{RW3} we have $L_{K^{ab}_\infty /\Qbar}= \Det(\lambda )$ with $\lambda \in \Lambda(G^{ab}_\infty )_\bullet$ the pseudomeasure of Serre.  The point is then that $\lambda = 2\wt\Theta ^{\,ab}$ with $\wt\Theta^{\,ab} \in \Lambda(G^{ab}_\infty )_\bullet$, which follows from Theorem~3.1b) of \cite{R}, because of Theorem~4.1 (loc.cit.) and the relation between $\lambda $ and $\mu _c$ discussed just after it.  Then $\wt L_{K^{ab}_\infty /\Qbar}=\;\Det(\wt\Theta ^{\,ab})$ and now the proof of the Corollary to Theorem~9 in \cite{RW3} shows that $\wt\Theta^{\,ab}$ is a unit of $\Lambda (G^{ab}_\infty ).$

\smallskip

\noindent b) {\it Claim:} $nr(1+2{\frak A}) = 1+4\Lambda (\Gamma)_\bullet.$
\begin{proof}[Proof of the claim]
If $x=a1+b\zeta _4 + c\wt s + d\zeta _4\wt s$ with $a,b,c,d\in \Lambda (\Gamma )_\bullet$, one computes $nr(x) = (a^2+b^2) - (c^2+d^2)$ from which $nr(1+2{\frak A})\subseteq 1+4\Lambda (\Gamma)_\bullet;$ equality follows from $nr\big((1+2a) + 2a\wt s\big) = (1+2a)^2-(2a)^2 = 1+4a$ for $a\in \Lambda (\Gamma )_\bullet\,.$
\end{proof}
Suppose first that the congruence for $\wt L_{K_\infty /\Qbar}(\alpha)$ holds.  Start with $\wt\Theta ^{\,ab}$ from a) in the lower left corner of the pullback square and map it to $\wt\Theta ^{\,ab} \mod 2$ in the lower right corner. Choosing any $y_0\in {\frak A}$ mapping to $\wt\Theta ^{\,ab} \mod 2$, we note that $y_0\in {\frak A}^\times$ because the maps in the pullback diagram are ring homomorphisms and the kernel $2{\frak A}$ of the right one is contained in the radical of ${\frak A}.$ Thus $nr(y_0)\in \Lambda (\Gamma )^\times_\bullet$ has $nr(y_0)^{-1}\wt L_{K_\infty /\Qbar}(\alpha ) \in 1+4\Lambda (\Gamma_\Qbar)_\bullet$ by the congruence, hence, by the Claim,  $nr(y_0)^{-1}\wt L_{K_\infty/\Qbar} (\alpha ) = nr(z)$, $z\in 1+2{\frak A }$. So $y_1 = y_0z$ is another lift of $\wt\Theta ^{\,ab} \mod 2$ and $nr(y_1) = \wt L_{K_\infty /\Qbar}(\alpha ).$ By the pullback diagram we get $Y\in \Lambda (G_\infty )^\times_\bullet$ which maps to $\wt \Theta ^{\,ab}$ and $y_1,$ where $nr(y_1) = \wt L_{K_\infty/\Qbar}(\alpha ).$

It follows that $\Det\, Y =\wt L_{K_\infty /\Qbar}.$ To see this we check that their values agree at every irreducible character $\chi$ of $G_\infty$; it even suffices to check it on the characters $1,\eta ,\nu,\eta \nu,\alpha $ of $G$ by Theorem~8 and Proposition~11 of \cite{RW2}, because every irreducible character of $G_\infty $ is obtained from these by multiplying by a character of type $W.$ It works for the characters $1,\eta ,\nu,\eta \nu$ of $G^{ab}_\infty$ by Proposition~12, 1b) (loc.cit.) since defl $(Y) = \wt \Theta ^{\,ab}$ and $\Det\;\wt\Theta ^{\,ab} = \wt L_{K^{ab}_\infty /\Qbar}$ by a). Finally, $(\Det\;Y)(\alpha ) = j_\alpha \big(nr(Y)\big) = nr(y_1) =\wt L_{K_\infty /\Qbar}(\alpha )$ by the commutative triangle before Theorem~8 (loc.cit.), the definition of $j_\alpha$, and $G_\infty =\Gamma \times H.$

The converse depends on related ingredients.  More precisely, $\wt L_{K_\infty  /\Qbar} \in \;\Det\,K_1\big((\Lambda  G_\infty)_\bullet\big)$ implies $\wt L_{K_\infty  /\Qbar} = \,\Det\;Y$ with $Y \in (\Lambda  G_\infty  )^\times_\bullet$ by surjectivity of $(\Lambda  G_\infty  )^\times _\bullet \to K_1\big((\Lambda  G_\infty  )_\bullet\big).$  Since $(\Lambda G^{ab}_\infty  )^\times _\bullet \to K_1\big((\Lambda G^{ab}_\infty)_\bullet\big)$ is an isomorphism, we get defl $Y = \wt \Theta^{\,ab}$ in $\Lambda  (G^{ab}_\infty  )^\times.$ Letting $y_1\in {\frak A}^\times$ be the image of $Y$ in the pullback diagram, it follows that $nr(y_1) =\wt L_{K_\infty/\Qbar}(\alpha  )$ and that $y_1$ maps to $\wt\Theta^{\,ab}\mod 2$ in $\Lambda  (G^{ab}_\infty  )_\bullet/2\Lambda  (G^{ab}_\infty)_\bullet\,.$  Given any $y$ as in b), then $y^{-1}_1 y$ maps to $1$ hence is in $1+2{\frak A}$ and our congruence follows from the Claim on applying $nr.$
\end{proof}

\section{Rewriting the congruence in testable form}\label{sec:rewriting}
Set $F_0 = \dfrac{\wt L_{K_\infty  /F,S}(1) + \wt L_{K_\infty/F,S}(\beta  ^2)}{2} - \wt L_{K_\infty  /F,S}(\beta  )\,.$

\begin{proposition} \label{Prop2}
\begin{enumerate}
\item[\rm a)] $F_0$ is in $\Lambda  (\Gamma  _\Qbar)_\bullet$
\item[\rm b)] $\wt L_{K_\infty  /\Qbar} \in \;\Det\;K_1\big(\Lambda(G_\infty  )_\bullet\big)$ if, and only if, $F_0 \in 4\Lambda (\Gamma  _\Qbar)_\bullet$
\end{enumerate}
\end{proposition}

\begin{proof}
Note that $\ind^G_C \, 1_C = 1_G +\eta ,$ \; $\ind^G_C\, \beta^2=\nu+\eta \nu,$ \; $\ind^G_C\,\beta =\alpha .$ When we inflate $\beta $ to a character of $\Gal(K_\infty /F)$ then $\ind^{G_\infty}_{\Gal(K_\infty /F)} \beta =\alpha $ with $\alpha$ inflated to $G_\infty ,$ etc.

By  Proposition 3.1 of the previous section we can write $\wt L_{K^{ab}_\infty /\Qbar} = \;\Det\,(\wt\Theta^{\,ab})$ with
$$
\wt\Theta^{\,ab} = a+bt^{ab} + cs^{ab} + ds^{ab}t^{ab}
$$
for some $a,b,c,d$  in $\Lambda (\Gamma )_\bullet\,.$ It follows that
$$
\begin{aligned}
\wt L_{K_\infty  /\Qbar}(1) &= a+b +c+d\\
\wt L_{K_\infty  /\Qbar}(\eta  ) &= a+b -c-d\\
\wt L_{K_\infty  /\Qbar}(\nu) &= a-b +c-d\\
\wt L_{K_\infty  /\Qbar}(\eta  \nu) &= a-b -c+d.
\end{aligned}
$$
  
Form $y = a+b\zeta _4 + c\wt s + d\zeta _4\wt s$ in $\big(\Lambda (\Gamma )_\bullet (\zeta _4)\big)*\la s\ra\,.$ By the computation in the Claim in the proof of Proposition~\ref{Prop1}, we have
$$
\begin{aligned}
nr(y) &= (a+c)(a-c) + (b+d)(b-d) \\
&= \frac{\wt L_\Qbar(1) +\wt L_{\Qbar}(\nu)}{2} \; \frac{\wt L_\Qbar(\eta  ) +\wt L_\Qbar(\eta  \nu)}{2} +\frac{\wt L_\Qbar(1) - \wt L_\Qbar(\nu)}{2} \; \frac{\wt L_{\Qbar}(\eta  ) - \wt L_\Qbar(\eta\nu)}{2} \\
&= \frac 14\,\big(\wt L_\Qbar(1+\eta  ) +\wt L_\Qbar (1+\eta  \nu)+\wt L_\Qbar (\nu+\eta) +\wt L_\Qbar (\nu + \eta  \nu)\big) \\
& \q +\frac 14 \big(\wt L_\Qbar(1+\eta  ) - \wt L_\Qbar(1+\eta  \nu) - \wt L_\Qbar (\nu+\eta  ) + \wt L_\Qbar (\nu+\eta  \nu)\big) \\
& = \frac{\wt L _\Qbar (1+\eta  ) +\wt L_\Qbar (\nu+\eta  \nu)}{2} = \frac{\wt L_F (1) + \wt L_F(\beta  ^2)}{2}\;, 
\end{aligned}
$$
because
$$
\wt L_{K_\infty  /\Qbar}(\ind^{G_\infty  }_{_{\Gal(K_\infty  /F)}}\chi) = \wt L_{K_\infty  /F}(\chi)
$$
for all characters $\chi  $ of $\Gal(K_\infty  /F).$  Thus also $\wt L_{K_\infty  /\Qbar}(\alpha  ) = \wt L_{K_\infty  /F}(\beta  )$, so we now have shown that
$$
F_0 = nr(y) - L_{K_\infty  /\Qbar}(\alpha)
$$
proving a), since $\wt L_{K_\infty  /F}(\beta  )\in (\Lambda \Gamma_F)_\bullet$ by \S2, as $\beta  $ has degree $1.$

Moreover, the image of $y$ under the right arrow of the pullback diagram of \S3 equals $\wt\Theta^{\,ab}\mod 2,$ by construction, hence b) follows directly from Proposition 3.1b).
\end{proof}

\begin{remark}
Considering $F_0$ in $\Lambda  (\Gamma  _\Qbar)_\bullet,$ instead of its natural home $\Lambda  (\Gamma  _F)_\bullet\,,$ is done to be consistent with the identification in b) of Proposition~\ref{Prop1}, via the natural isomorphisms $\Gamma  {\to} \Gamma  _F{\to} \Gamma  _\Qbar$: this is the sense in which $L_{K_\infty/\Qbar}(\alpha  ) = L_{K_\infty/F}(\beta).$ 
\end{remark}

The congruence $F_0\equiv 0 \mod 4\Lambda (\Gamma _\Qbar)_\bullet$ can now be put in the more testable form of Conjecture 1.  Let $\gamma _\Qbar$ be the generator of $\Gamma _\Qbar$ which, when extended to $\Qbar(\sqrt{-1}\,)$ as the identity, acts on all $2$-power roots of unity in $\Qbar_\infty(\sqrt{-1}\,)$ by raising them to the $u^{\text{\rm th}}$ power, where $u\equiv 5 \mod 8\Z_2$ as fixed before.  Then the Iwasawa isomorphism $\Lambda (\Gamma _\Qbar)\simeq \Z_2\big[[T]\big],$ under which $\gamma_\Qbar -1$ corresponds to $T,$ makes $ F_0\in \Lambda (\Gamma_\Qbar)_\bullet$ correspond to some $F_0(T) \in \Z_1\big[[T]\big]_\bullet$ and the congruence of Proposition 4.1b) to
$$
F_0 (T) \equiv 0 \mod 4\Z_2\big[[T]\big]_\bullet
$$
Since $\beta $ is an abelian character, we know that $\wt L_{F,S}(\beta ^2),$\, $\wt L_{F,S}(\beta )$ correspond to elements of $\Z_2\big[[T]\big],$ not just $\Z_2\big[[T]\big]_\bullet$ (cf.  \S4 of \cite {RW2}), and $\wt L_F(1)$ to one of $T^{-1}\Z_2\big[[T]\big]$. We thus have
$$
F_0(T) =\frac{x_0}{T} +\sum^\infty  _{j=1} x_j T^{j-1}
$$
with $x_j\in \Z_2$ for all $j\ge 0.$

By the interpolation definition of $\big(\wt L_{F,S}(\beta^i)\big)(T)$ (cf \S4 of \cite{R}), it follows that
$$
F_0(u^s-1) =\frac 12\,\Big(\frac{L_{F,S}(1-s,1)}{4}+\frac{L_{F,S}(1-s,\beta ^2)}{4} - 2\;\frac{L_{F,S}(1-s,\beta)}{4}\Big) = f_0(1-s)\,.
$$
This implies
$$
x_0 = \,-\;\frac{\rho_{F,S}\log(u)}{8}\;,
$$
because the left side  is 
$$
\us{T\to 0}\lim\; TF_0(T) = \us{s\to 1}\lim\;\frac{u^{1-s}-1}{s-1}\; (s-1)f_0(s) = -\log(u)\,\us{s\to 1}\lim \,(s-1)\;\frac{L_{F,S}(s,1)}{8}
$$
as required. Note that $u\equiv 5 \mod 8$ implies that $\dfrac{\log(u)}{4}$ is a $2$-adic unit, hence $\frac 12\,\rho_{F,S} \in \Z_2$ is in $4\Z_2$ if, and only if, $x_0 \in 4\Z_2.$

\smallskip

Define $F_1(T) = F_0 (T) - x_0 T^{-1} =\os\infty {\us{j=1}\sum} x_jT^{j-1}\in \Z_2\big[[T]\big].$ It follows that
$$
F_1(u^{s}-1) = -\,\frac{x_0}{u^{s}-1} + F_0(u^{s}-1) = \frac{\rho_{F,S}\log(u)}{8(u^{s}-1)} + f_0(1-s)
$$
which is $f_1(1-s), $ with $f_1$ as in \S1, thus reconciling the notation $F_1(T)$ here with that there.  Thus Conjecture~\ref{conj1} of \S1 is equivalent to Conjecture~\ref{conj2} of \S2 for the special case $K_\infty /\Qbar$ of \S1.

\section{Testing Conjecture~\ref{conj1}}

Let $\chi$ be a $2$-adic character of the Galois group $C$ of $K/F$ and let $\mathfrak{f}$ be the conductor of $K/F$. By class field theory, we view $\chi$ as a map on the group of ideals relatively prime to $\mathfrak{f}$. Fix a prime ideal $\mathfrak{c}$ not dividing $\mathfrak{f}.$ For $\mathfrak{a}$, a fractional ideal relatively prime to $\mathfrak{c}$
and $\mathfrak{f}$, let $\mathcal{Z}_{\mathfrak{f}}(\mathfrak{a}, \mathfrak{c}; s)$ denote the associated $2$-adic twisted partial zeta function [PCN]. Thus, we have
\begin{equation*}
L_{F,S}(s, \chi) = \frac{1}{\chi(\mathfrak{c}) \langle N\mathfrak{c} \rangle^{1-s} - 1} \prod_{\mathfrak{p}} \left(1 - \frac{\chi(\mathfrak{p})}{N\mathfrak{p}} \langle N\mathfrak{p}\rangle^{1-s}\right) \sum_{\sigma \in G} \chi(\sigma)^{-1} \mathcal{Z}_\mathfrak{f}(\mathfrak{a}_\sigma^{-1}, \mathfrak{c}; s)   
\end{equation*}
where $\mathfrak{p}$ runs through the prime ideals of $F$ in $S$ not dividing $2\mathfrak{f}$, $\mathfrak{a}_\sigma$ is a (fixed) integral ideal coprime with $2\mathfrak{fc}$ whose Artin symbol is $\sigma$

\smallskip

Denote the ring of integers of $F$ by $\mathcal{O}_F$ and let $\gamma \in \mathcal{O}_F$ be such that $\mathcal{O}_F = \mathbb{Z} + \gamma \mathbb{Z}$. In [Rob] (see also [BBJR] for a slightly different presentation), it is shown that the function $\mathcal{Z}_{\mathfrak{f}}(\mathfrak{a}, \mathfrak{c}; s)$ is defined by the following integral
$$
\mathcal{Z}_\mathfrak{f}(\mathfrak{a}, \mathfrak{c}; s) = \int
\frac{\langle N\mathfrak{a} \, N(x_1 + x_2\gamma)\rangle^{1-s}}{N\mathfrak{a} \, N(x_1 + x_2\gamma)} \,
d\mu_{\mathfrak{a}}(x_1, x_2)
$$
where the integration domain is $\mathbb{Z}_2^2$, $\langle\,\rangle$ is extended to $\mathbb{Z}_2$ by $\langle x\rangle = 0$ if $x \in 2\mathbb{Z}_2$, and the measure $\mu_{\mathfrak{a}}$ is a measure of norm $1$ (depending also on $\gamma$, $\mathfrak{f}$ and $\mathfrak{c}$). 

\smallskip

Assume now, as we can do without loss of generality, that the ideal $\mathfrak{c}$ is such that $\langle N\mathfrak{c} \rangle \equiv 5 \pmod{8\mathbb{Z}_2}$ and take $u = \langle N\mathfrak{c}\rangle$. For $s \in \mathbb{Z}_2$, we let $t = t(s) = u^s - 1 \in 4\mathbb{Z}_2$, so that $s = \log(1+t)/\log(u)$. For $x \in \mathbb{Z}_2^\times$, one can check readily that
$$
\langle x \rangle^s = \left(u^{\mathcal{L}(x)}\right)^s = (1 + u^s - 1)^{\mathcal{L}(x)} = \sum_{n \geq 0} \binom{\mathcal{L}(x)}{n} t^n 
$$
where $\mathcal{L}(x) = \log\langle x\rangle/\log u \in \mathbb{Z}_2$. For $x \in \mathbb{Z}_2^\times$, we set
$$
L(x;T) = \sum_{n \geq 0} \binom{\mathcal{L}(x)}{n} T^n \in \mathbb{Z}_2[[T]]
$$
and $L(x; T) = 0$ if $x \in 2\mathbb{Z}_2$. Now, we define
\begin{align*}
R(\mathfrak{a}, \mathfrak{c}; T) & = \int \frac{ L\big(N\mathfrak{a} \, N(x_1 + x_2\gamma); T\big) }{N\mathfrak{a} \, N(x_1 + x_2\gamma)}\, d\mu_{\mathfrak{a}}(x_1,x_2) \in \mathbb{Z}_2[[T]] \\[5pt]
B(\chi; T) & = \chi(\mathfrak{c})(T + 1) - 1 \in \mathbb{Z}_2[\chi][T], \\[5pt] 
A(\chi; T) & = \prod_{\mathfrak{p}} \left(1 - \frac{\chi(\mathfrak{p})}{N\mathfrak{p}} L(N\mathfrak{p}; T)\right) \sum_{\sigma \in G} \chi(\sigma)^{-1} R(\mathfrak{a}_\sigma^{-1}, \mathfrak{c}; T) \in \mathbb{Z}_2[\chi][[T]]
\end{align*}
where $\mathfrak{p}$ runs through the prime ideals of $F$ in $S$ not dividing $2\mathfrak{f}$.

\begin{proposition}\label{interpol}
We have, for all $s \in \mathbb{Z}_2$
\begin{equation*}
L_{F,S}(1-s, \chi) = \frac{A(\chi; u^s-1)}{B(\chi; u^s-1)}.
\end{equation*}
\end{proposition}
We now specialize to our situation. For that, we need to make the additional assumption that $\beta^2(\mathfrak{c}) = -1$, so $\beta(\mathfrak{c})$ is a fourth root of unity in $\Qbar_2^c$ that we will denote by $i$. Thus, we have
\begin{align*}
& B(1; T) = T,\ B(\beta; T) = i(T+1)-1, \\
& B(\beta^2; T) = -T-2,\ B(\beta^3; T) = -i(T+1)-1. 
\end{align*}
Let  $x \mapsto \bar{x}$ be the $\Qbar_2$-automorphism of $\Qbar_2(i)$ sending $i$ to $-i$. Then we have $\overline{L_{F,S}(1-s, \beta)} = L_{F,S}(1-s, \beta^3)$ by the expression of $L_{F,S}(s, \chi)$ given at the beginning of the section since the twisted partial zeta functions have values of $\Qbar_2$ and $\bar\beta = \beta^3$. And furthermore, 
$$
L_{F,S}(s, \beta^3) = L_{\Qbar,S}(s, \text{Ind}_C^G(\beta^3)) =  L_{\Qbar,S}(s, \text{Ind}_C^G(\beta)) = L_{F,S}(s, \beta).
$$
Therefore, by Prop.~\ref{interpol}, we deduce that
\begin{align*}
A(\beta; u^s-1) + \bar A(\beta; u^s-1) & =  \big(B(\beta; T) + B(\beta^3; T)\big) L_{F,S}(1-s, \beta) \\
& =  -2 L_{F,S}(1-s, \beta).
\end{align*}
Since 
$$
f_1(s) = \frac{\rho_{F,S} \log u}{8(u^{1-s}-1)} + \frac{1}{8}\left(L_{F,S}(s,1) + L_{F,S}(s, \beta^2) - 2L_{F,S}(s, \beta)\right)
$$
we find that
\begin{equation*}
F_1(T) = \frac{\rho_{F,S} \log u}{8T} + \frac{1}{8} \left(\frac{A(1; T)}{T} - \frac{A(\beta^2; T)}{T+2} + A(\beta; T) + \bar{A}(\beta, T)\right)
\end{equation*}
is such that $F_1(u^n-1) = f_1(1-n)$ for $n = 1, 2, 3, \dots$. 

\smallskip

The conjecture that we wish to check states that
$$
\frac{1}{2}\rho_{F,S} \in 4\mathbb{Z}_2 \quad\text{and}\quad F_1(T) \in 4\mathbb{Z}_2[[T]].  
$$
Now define $D(T) = 8T(T+2) F_1(T)$, so that
\begin{multline*}
D(T) = (T+2) \left(\rho_{F,S} \log u + A(1; T)\right) \\
	- T A(\beta^2;T) + T(T+2)\left(A(\beta; T) + \bar{A}(\beta, T)\right).
\end{multline*}
We can now give a final reformulation of the conjecture which is the one that we actually tested. 
\begin{conj}
$$
\rho_{F,S} \in 8\mathbb{Z}_2 \quad\text{and}\quad D(T) \in 32\mathbb{Z}_2[[T]] 
$$
\end{conj}
The computation of $\rho_{F,S}$ is done using the following formula [Col]
$$
\rho_{F,S} = 2 \, h_F \, R_F d_F^{-1/2} \, \prod_{\fP}\left(1 -
  1/N(\fP)\right) 
$$
where $h_F$, $R_F$, $d_F$ are respectively the class number, $2$-adic regulator and discriminant of $F$ and $\fP$ runs through all primes of $F$ above $2$. Note that although $R_F$ and $d_F^{-1/2}$ are only defined up to sign, the quantity $R_Fd_F^{-1/2}$ is uniquely determined in the following way: Let $\iota$ be the embedding of $F$ into $\mathbb{R}$ for which $\sqrt{d_F}$ is positive and let $\varepsilon$ be the fundamental unit of $F$ such that $\iota(\varepsilon) > 1$. Then for any embedding $g$ of $F$ into $\Qbar^{\,c}_2$, we have
$$
R_F d_F^{-1/2} = \frac{\log_2 g(\varepsilon)}{g(\sqrt{d})}\,.
$$

Now, for the computation of $D(T)$, the only difficult part is the computations of the $R(\mathfrak{a}, \mathfrak{c}; T)$. The measures $\mu_{\mathfrak{a}}$ are computed explicitly using the methods of [Rob] (see also [BBJR]), that is we construct a power series $M_{\mathfrak{a}}(X_1, X_2)$ in $\Qbar_2[X_1, X_2]$ with integral coefficients, such that
$$
\int (1+t_1)^{x_1} (1+t_2)^{x_2} \, d\mu_{\mathfrak{A}}(x_1, x_2) = M_{\mathfrak{A}}(t_1, t_2) \quad \text{ for all } t_1, t_2 \in 2\mathbb{Z}_2. 
$$
In particular, if $f$ is a continuous function on $\mathbb{Z}_2^2$ with values in $\Cbar_2$ and Mahler expansion
$$
f(x_1, x_2) = \sum_{n_1, n_2 \geq 0} f_{n_1, n_2} \binom{x_1}{n_1}\binom{x_2}{n_2} 
$$
then we have
$$
\int f(x_1, x_2) \, d\mu_{\mathfrak{A}}(x_1, x_2) = \sum_{n_1, n_2 \geq 0} f_{n_1, n_2} m_{n_1,n2} 
$$
where $M_{\mathfrak{A}}(X_1, X_2) = \sum\limits_{n_1, n_2 \geq 0} m_{n_1, n_2} X_1^{n_1} X_2^{n_2}$.

We compute this way the first few coefficients of the power series $A(\chi; T)$, for $\chi = \beta^j$, $j = 0,1,2,3$, and then deduce the first coefficients of $D(T)$ to see if they do indeed belong to $32\mathbb{Z}_2[[T]]$. We found that this was indeed always the case; see next section for more details. 

To conclude this section, we remark that, in fact, we do not need the above formula to compute $\rho_{F,S}$ since the constant coefficient of $A(1;T)$ is $-\rho_{F,S} \, \log u$. (This can be seen directly from the expression of $x_0$ given at the end of Section~\ref{sec:rewriting} or using the fact that $D(T)$ has zero constant coefficient since $F_1(T) \in \mathbb{Z}_2[[T]]$.) However, we did compute it using this formula since it then provides a neat way to check that (at least one coefficient of) $A(1;T)$ is correct. 

\section{The numerical verifications}

We have tested the conjecture in $60$ examples. The examples are separated in three subcases of $20$ examples according to the way $2$ decomposes in the quadratic subfield $F$: ramified, split or inert. In each subcase, the examples are actually the first $20$ extensions $K/\mathbb{Q}$ of the suitable form of the smallest discriminant. These are given in the following three tables of Figure 2 where the entries are: the discriminant $d_F$ of $F$, the conductor $\mathfrak{f}$ of $K/F$ (which is always a rational integer) and the discriminant $d_K$ of $K$. In each example, we have computed $\rho_{F,S}$ and the first $30$ coefficients of $D(T)$ to a precision of at least $2^8$ and checked that they satisfy the conjecture.

\begin{scriptsize}
\begin{figure}[htp]
\begin{tabular}{ccc}
$\begin{array}{r|r|r}
\multicolumn{3}{c}{\text{$2$ ramified in $F$}} \\
\hline
d_F & \mathfrak{f} & d_K \\
\hline
   44    &   3 &       2\,732\,361\,984 \\
   156  &   2 &       9\,475\,854\,336 \\
   220  &   2 &     37\,480\,960\,000 \\
   12    & 14 &     39\,033\,114\,624 \\
   156  &   4 &    151\,613\,669\,376 \\
   380  &   2 &    333\,621\,760\,000 \\
   152  &   3 &    389\,136\,420\,864 \\
    24   & 11 &    587\,761\,422\,336 \\
   876  &   1 &    588\,865\,925\,376 \\
   220  &   4 &    599\,695\,360\,000 \\
   444  &   2 &    621\,801\,639\,936 \\
    12   & 28 &    624\,529\,833\,984 \\
    44   & 12 &    699\,484\,667\,904 \\
    92   &   6 &    835\,600\,748\,544 \\
    60   &   8 &    849\,346\,560\,000 \\
    44   & 10 &    937\,024\,000\,000 \\
    12   & 19 &    975\,543\,388\,416 \\
    12   & 26 &   1\,601\,419\,382\,784 \\
    44   & 15 &   1\,707\,726\,240\,000 \\
 1\,164 &  1 &   1\,835\,743\,170\,816 
\end{array}
$ & $
\begin{array}{r|r|r}
\multicolumn{3}{c}{\text{$2$ inert in $F$}} \\
\hline
d_F & \mathfrak{f} & d_K \\
\hline
  445  &    1 &        39\,213\,900\,625 \\
   5    &   21 &        53\,603\,825\,625 \\
  205  &    3 &       143\,054\,150\,625 \\
  221  &    3 &       193\,220\,905\,761 \\
   61   &    5 &       216\,341\,265\,625 \\
  205  &    4 &       452\,121\,760\,000 \\
  221  &    4 &       610\,673\,479\,936 \\
  901  &    1 &       659\,020\,863\,601 \\
   29   &  15 &       895\,152\,515\,625 \\
1\,045 &   1 &    1\,192\,518\,600\,625 \\
    5    &  16 &    1\,911\,029\,760\,000 \\
   109  &   5 &    2\,205\,596\,265\,625 \\
1\,221 &   1 &    2\,222\,606\,887\,281 \\
    29   & 20 &    2\,829\,124\,000\,000 \\
    29   & 13 &    3\,413\,910\,296\,329 \\
   205  &   7 &    4\,240\,407\,600\,625 \\
   149  &   5 &    7\,701\,318\,765\,625 \\
1\,677 &   1 &    7\,909\,194\,404\,241 \\
    21   & 19 &    9\,149\,529\,982\,761 \\
   341  &   3 &    9\,857\,006\,530\,569 \\
\end{array}
$ & $
\begin{array}{r|r|r}
\multicolumn{3}{c}{\text{$2$ split in $F$}} \\
\hline
d_F & \mathfrak{f} & d_K \\
\hline
   145  &   1 &              442\,050\,625 \\
    41   &   5 &         44\,152\,515\,625 \\
   505  &   1 &         65\,037\,750\,625 \\
   689  &   1 &        225\,360\,027\,841 \\
   777  &   1 &        364\,488\,705\,441 \\
   793  &   1 &        395\,451\,064\,801 \\
    17   & 13 &        403\,139\,914\,489 \\
   897  &   1 &        647\,395\,642\,881 \\
   905  &   1 &        670\,801\,950\,625 \\
   305  &   3 &        700\,945\,700\,625 \\
   377  &   3 &     1\,636\,252\,863\,921 \\
1\,145 &   1 &     1\,718\,786\,550\,625 \\
   145  &   8 &     1\,810\,639\,360\,000 \\
   305  &   4 &     2\,215\,334\,560\,000 \\
1\,313 &   1 &     2\,972\,069\,112\,961 \\
   377  &   4 &     5\,171\,367\,076\,096 \\
   545  &   3 &     7\,146\,131\,900\,625 \\
    17   & 21 &     7\,163\,272\,192\,041 \\
1\,705 &   1 &     8\,450\,794\,350\,625 \\
   329  &   3 &     8\,541\,047\,165\,049 \\
\end{array}
$
\end{tabular}
\caption{}
\end{figure}
\end{scriptsize}

We now give an example, namely the smallest example for the discriminant of $K$. We have $F = \Qbar(\sqrt{145})$
and $K$ is the Hilbert class field of $F$. The prime $2$ is split in $F/\Qbar$ and the primes above $2$ in $F$ are inert in $K/F$. We compute $\rho_{F,S}$ and find that
$$
\rho_{F,S} \equiv 2^7 \pmod{2^8}
$$

Using the method of the previous section, we compute the first $30$ coefficients of the power series $A(\,\cdot\,;T)$ to a $2$-adic precision of $2^8$. We get
\begin{align*}
A(1; T) \equiv 2^2 \big( & 16T + 57T^3 + 44T^4 + 8T^5 + 40T^6 + 21T^7 + 40T^8 + 30T^9 \\ 
  & + 16T^{10} + 49T^{11} + 56T^{12} + 29T^{13} + 32T^{14} + 50T^{15}  \\
  & + 62T^{16} + 47T^{17} + 48T^{18} + 60T^{19} + 32T^{20} + 16T^{21}  \\
  & + 8T^{22} + 21T^{23} + 30T^{24} + 26T^{25} + 2T^{26} + 9T^{27} \\
  & + 56T^{28} + 34T^{29} \big) + O(T^{30}) \pmod{2^8}
\end{align*}

\begin{align*}
  A(\beta; T) \equiv 2^2 \big( & (28 + 1124i) + (36 + 1728i)T + (47 + 45i)T^2 + (56 + 153i)T^3 \\ 
  & + (46 + 154i)T^4 + (56 + 282i)T^5 + (55 + 433i)T^6 \\
  & + (54 + 435i)T^7 + (40 + 386i)T^8 + (48 + 392i)T^9 \\
  & + (63 + 65i)T^{10} + (48 + 257i)T^{11} + (63 + 161i)T^{12} \\
  & + (20 + 477i)T^{13} + (38 + 182i)T^{14} + (56 + 66i)T^{15} \\
  & + (37 + 35i)T^{16} + (6 + 341i)T^{17} + (20 + 446i)T^{18} \\ 
  & + (40 + 412i)T^{19} + 368iT^{20} + (56 + 336i)T^{21} \\
  & + (61 + 291i)T^{22} + (40 + 427i)T^{23} + (34 + 38i)T^{24} \\
  & + (48 + 94i)T^{25} + (9 + 47i)T^{26} + (6 + 497i)T^{27} \\
  & + (40 + 42i)T^{28} + (44 + 52i)T^{29} \big) + O(T^{30}) \pmod{2^8}
\end{align*}

\begin{align*}
  A(\beta^2; T) \equiv 2^2 \big( & 32 + 32T + 22T^2 + 39T^3 + 36T^4 + 20T^5 + 62T^6 + 27T^7 \\ 
  & + 16T^8 + 62T^9 + 46T^{10} + 23T^{11} + 30T^{12} + 51T^{13} \\
  & + 4T^{14} + 2T^{15} + 56T^{16} + 33T^{17} + 44T^{18} + 12T^{19} \\
  & + 40T^{20} + 8T^{21} + 54T^{22} + 11T^{23} + 34T^{24} + 42T^{25} \\
  & + 43T^{27} + 56T^{28} + 46T^{29} \big) + O(T^{30}) \pmod{2^8}
\end{align*}

Therefore 
\begin{align*}
  D(T) \equiv 2^5 \big( & 6T + 7T^2 + 4T^3 + 5T^4 + 4T^7 + 2T^8 + 4T^9 + 2T^{10} + 4T^{11} \\ 
  & + T^{12} + 6T^{13} + 7T^{14} + 3T^{16} + 5T^{17} + 2T^{18} + 3T^{19} \\ 
  & + 7T^{20} + 5T^{21} + 7T^{22} + 4T^{23} + 4T^{24} + T^{25} + 7T^{26} \\ 
  & + 3T^{27} + 7T^{28} + 6T^{29} \big) + O(T^{30}) \pmod{2^8}
\end{align*}
and the conjecture is satisfied by the first $30$ coefficients of the series $D$ associated to the extension.

\smallskip

Note, as a final remark, that we have tested the conjecture in the same way for $30$ additional examples where $F$ is real quadratic, $K/F$ is cyclic of order $4$ but $K$ is not a dihedral extension of $\Qbar$ (either $K/\Qbar$ is not Galois or its Galois group is not the dihedral group of order $8$). In all of these examples, we found that the conjecture was not satisfied, that is either $\rho_{F,S}$ did not belong to $8\mathbb{Z}_2$ or one of the first $30$ coefficients of the associated power series $D$ did not belong to $32\mathbb{Z}_2$.

\end{document}